\documentclass[letterpaper,oneside,11pt]{article}
\usepackage[english]{babel}
\usepackage[T1]{fontenc}
\usepackage[utf8]{inputenc}
\usepackage{amsmath, amssymb, amsthm, amsfonts}
\usepackage[abbrev]{amsrefs}
\usepackage{mathtools}
\mathtoolsset{showonlyrefs}
\usepackage{parskip}
\usepackage{xcolor}
\usepackage[normalem]{ulem}

\newtheorem{teo}{Theorem}[section]
\newtheorem{prop}[teo]{Proposition}
\newtheorem{lemma}[teo]{Lemma}
\newtheorem{re}[teo]{Remark}

\numberwithin{equation}{section}

\newcommand{\R}{\mathbb{R}}
\newcommand{\C}{\mathbb{C}}
\newcommand{\qtq}[1]{\quad\text{#1}\quad}
\renewcommand{\Re}{\operatorname{Re}}
\renewcommand{\Im}{\operatorname{Im}}
\newcommand{\la}{\left\langle}
\newcommand{\ra}{\right\rangle}
\newcommand{\raa}{\right\rangle_{L^2}}
\newcommand{\rah}{\right\rangle_{\dot H^1}}

\usepackage{enumitem}

\usepackage[plainpages=false,pdfpagelabels,pdfusetitle,pdfdisplaydoctitle, pdfstartpage=1, pdfstartview={{XYZ null null 1.00}}]{hyperref}
\usepackage[top=3cm, bottom=2cm, left=3cm, right=2cm]{geometry}
\author{}

\begin{document}
\scrollmode
\title{Threshold solutions for the $3d$ cubic INLS: \\ the energy-critical case}

\author{Luccas Campos, Luiz Gustavo Farah, and Jason Murphy} 
\date{} 

\maketitle

\makeatletter{\renewcommand*{\@makefnmark}{}
\makeatother}
\makeatletter{\renewcommand*{\@makefnmark}{}
\makeatother}
\begin{abstract}\noindent
We study the energy-critical $3d$ cubic inhomogeneous NLS equation $i\partial_t u + \Delta u + |x|^{-1}|u|^2 u=0$.  In this work, we prove the existence of special solutions $W^\pm$ with energy equal to that of the ground state $W$ and use these solutions to characterize the behavior of solutions at the ground state energy. The singular factor $|x|^{-1}$ in the nonlinearity significantly limits the smoothness of the ground state and prompts a novel approach to the modulation analysis.
\end{abstract}

\section{Introduction}

We consider the initial-value problem for the following  nonlinear Schrödinger equation with focusing inhomogeneous nonlinearity in three space dimensions: 
\begin{equation}\label{NLS}
\begin{cases} i\partial_t u + \Delta u + |x|^{-1}|u|^2 u = 0, \\[2mm]
u|_{t=0} = u_0\in \dot H^1(\R^3).
\end{cases}
\end{equation}
Here $u:\R_t\times\R_x^3\to\C$. 

The critical space of initial data for \eqref{NLS} is $\dot H^1(\R^3)$, in the sense that this norm is invariant under the scaling symmetry that preserves the class of solutions to \eqref{NLS}, namely $u(t,x)\mapsto \lambda ^{\frac12}u(\lambda^2 t, \lambda x)$ with $\lambda >0$. This scaling also preserves the conserved quantity known as the \textit{energy}
\begin{equation}
E[u] = \tfrac{1}{2}\int |\nabla u|^2\,dx - \tfrac{1}{4}\int |x|^{-1}|u|^4\,dx,
\end{equation}
and for this reason we refer to \eqref{NLS} as an \emph{energy-critical} or \emph{$\dot H^1$-critical} model. 

More generally, for the inhomogeneous NLS on $\mathbb{R}^d$ having nonlinearity $|x|^{-b}|u|^p u$, the equation is $\dot H^{s_c}$-critical in this sense if
\[
s_c:=\tfrac{d}{2}-\tfrac{2-b}{p}.
\]  
Recently there has been a great deal of progress on classifying dynamics for focusing inhomogeneous NLS in the regime $s_c\in[0,1]$, corresponding to the mass-critical case ($s_c=0$), intercritical case ($s_c\in(0,1)$), and energy-critical case $s_c=1$ (see e.g. \cite{GM21, CHL20, CL21, MiaoMurphyZheng2021, Murphy2022, CamposCardoso2022, LiuMiaoZheng2025, GuzmanXu2025, CM_Threshold_INLS_2023, Liu2024arxiv}).  An important role is played by a special solution known as the \emph{ground state}, which will be discussed in more detail below.  In particular, the mass and/or energy of the ground state can be used to define a threshold below which one has a scattering/blowup dichotomy.  At the threshold, new solution behaviors are possible (e.g. heteroclinic orbits exhibiting scattering/blowup in one direction and converging to the ground state in the other).  Incorporating these dynamics, one can obtain a complete classification of dynamics at the ground state threshold, as well.  These results (as well as the techniques used to prove them) parallel those previously obtained for  the standard power-type NLS 
(see e.g. \cite{DM_Thre, DM08, LZ09, DR10, CFR22}).  Some new challenges arise in the setting of inhomogeneous NLS due to broken translation symmetry, although the localizing effect of $|x|^{-b}$ can be exploited to achieve certain simplifications.

In general, the problem of classifying threshold dynamics for \eqref{NLS} becomes more difficult as the parameter $b$ increases, as the singularity in the nonlinearity is reflected in singular behavior of the ground state at the origin. This difficulty manifests primarily in the context of the \emph{modulation analysis}.  This refers to the analysis of solutions at times when they are close to the ground state orbit, which involves introducing modulation parameters defined implicitly via certain orthogonality conditions.  An essential step in the analysis involves estimating the variation of these modulation parameters, and it is precisely in this step that the precise decay/regularity properties of the ground state play a key role. To date, most of the results on threshold classification currently available for the inhomogeneous NLS impose constraints on the parameter $b$ that essentially guarantee that the modulation analysis developed for the power-type NLS can be employed with minimal modifications (cf. \cite{CM_Threshold_INLS_2023, Liu2024arxiv}).  One exception is our companion paper \cite{CFMarxiv25}, which extends the analysis of \cite{CM_Threshold_INLS_2023} for the $3d$ cubic INLS from $b\in(0,\tfrac12)$ to the full energy-subcritical range $b\in(0,1)$, precisely by refining the approach to the modulation analysis.  However, the methods introduced in \cite{CFMarxiv25} fall short of treating the energy-critical case $b=1$.  The contribution of the present work is to demonstrate an approach to the modulation analysis that suffices to handle even more singular regimes.

In the energy-critical case of inhomogeneous NLS, Liu, Yang, and Zhang \cite{Liu2024arxiv} previously classified threshold dynamics in dimensions $d\in\{3,4,5\}$ under the constraint $0<b<1-\tfrac12(d-4)^2$.  Roughly speaking, this is the regime in which the standard modulation analysis goes through (see \cite[Remark~1.9]{Liu2024arxiv} for a more detailed discussion).  In three dimensions the restriction is $b\in(0,\tfrac12)$.  As mentioned above, our companion paper \cite{CFMarxiv25} demonstrates a method for treating energy-subcritical nonlinearities allowing for $b\in(0,1)$.  The present work provides a new approach to the modulation analysis that can treat the energy-critical cubic nonlinearity $|x|^{-1} |u|^2 u$ in three space dimensions.  We expect that this approach can be extended to other singular settings, as well. 

Before we can state our main result, we must first properly introduce the ground state solution for \eqref{NLS}.  In the case of \eqref{NLS}, the ground state is given explicitly by 
\[
W(x) = \frac{1}{1+\tfrac{1}{2}|x|}.
\]
The function $W\in\dot H^1$ is the unique optimizer (up to symmetries) of the sharp Sobolev embedding 
\begin{equation}\label{weighted-sobolev}
\int |x|^{-1}|f|^4\,dx \leq \tfrac{3}{8 \pi}\left(\int |\nabla f|^2\,dx\right)^2.
\end{equation}
The function $W$ also satisfies the nonlinear elliptic equation 
\[
\Delta W + |x|^{-1}|W|^2W = 0,
\]
so that $W$ is also a static solution to \eqref{NLS}. 
%
%
%
%
%

As mentioned above, the behavior of solutions with energy below the ground-state energy is well understood \cite{GM21, CHL20, CL21}. Our interest lies in describing the dynamics of solutions precisely at the threshold, that is, solutions $u$ for which
\begin{equation}\label{threshold}
E[u] = E[W] = \tfrac{2\pi}{3}.
\end{equation}

Our first main result asserts the existence of two special threshold solutions (heteroclinic orbits).  Using these two new solutions, we then classify all possible solution dynamics at the threshold energy.  Our main results are the following. 

\begin{teo}[Existence of special solutions]\label{thm:Qpm}
	There exist forward-global radial solutions $W^{\pm}$ to \eqref{NLS} with
	\[	
	E[W^{\pm}] = E[W],
	\]
	satisfying
	\[
	\|W^{\pm}(t) - W\|_{\dot H^{1}} \lesssim e^{-ct},
	\]
	for some $c>0$ and all $t>0$.
	
	The solution $W^{+}$ satisfies
	\[
	\|W^{+}(0)\|_{\dot H^{1}} > \| W\|_{\dot H^{1}},
	\]
	while the solution $W^{-}$ is global, satisfies
	\[
	\| W^{-}(0)\|_{\dot H^{1}} < \| W\|_{\dot H^{1}},
	\]
	and scatters in $\dot H^{1}(\mathbb{R}^{3})$ as $t\to -\infty$.
\end{teo}

\begin{re} It is expected that $W^+$ should blow up in finite negative time.  However, we are presently unable to prove this due to the fact that $W^+\notin L^2(\mathbb{R}^{3})$.
\end{re}

\begin{teo}[Classification of threshold dynamics]\label{thm:threshold}
	Let $u_{0}\in \dot H^{1}(\mathbb{R}^3)$ satisfy $E[u_0] = E[W]$ and let $u$ be the corresponding solution of \eqref{NLS}. We have the following
	
	\begin{itemize}
		\item[(i)] If
		\[
	 \|u_{0}\|_{\dot H^{1}}
		< \| W\|_{\dot H^{1}},
		\]
		then $u$ either scatters as $t\to\pm\infty$ or $u=W^{-}$ up to the symmetries of the equation.
		
		\item[(ii)] If
		\[
	 \| u_{0}\|_{\dot H^{1}}
		=  \| W\|_{\dot H^{1}},
		\]
		then $u = W$ up to the symmetries of the equation.
		
		\item[(iii)] If
		\[
		 \| u_{0}\|_{\dot H^{1}}
		>  \| W\|_{\dot H^{1}}
		\]
		and $u_{0}$ is radial, then $u$ either blows up in finite positive and negative times or $u=W^{+}$ up to the symmetries of the equation.
	\end{itemize}
\end{teo}

\begin{re} Solutions as in \textit{(i)} will be called \emph{constrained}; solutions as in \textit{(iii)} will be called \emph{unconstrained}. 
\end{re}

As mentioned above, this theorem parallels the results for the standard power-type NLS \cite{DM_Thre, DM08, LZ09, DR10, CFR22}, as well as some existing results for inhomogeneous NLS \cite{CM_Threshold_INLS_2023, Liu2024arxiv}.  The main new challenge addressed in this paper is the presence of the factor $|x|^{-1}$, which is too singular to be addressed directly by any existing approaches.  As described briefly above, the primary technical challenge arises in the setting of the modulation analysis, stemming from the lack of regularity of $W$.  Our idea to resolve this issue is ultimately straightforward, but takes some care to implement properly.  Essentially, we modify the modulation parameters to impose orthogonality not to the ground state itself, but to a low frequency projection thereof.  This provides enough regularity to control the modulation parameters properly, but it is not clear \emph{a priori} if such orthogonality conditions are sufficient to access certain coercivity estimates for the operator $\mathcal{L}$ that arises from linearization around the ground state.  Our approach therefore requires a refined characterization of the negative and null directions for $\mathcal{L}$ in order to properly track the error terms arising from frequency projection.

Understanding the spectral properties of the linearized operator leads to some technicalities, as well. Indeed, the commutator asymptotics $[\Delta; |x|^{-1}W^2]f \sim |x|^{-3}f$ as $|x| \to 0$ destroy integrability around the origin, thus complicating the study of the eigenvalues of $\mathcal L$.  Fortunately, such difficulties have recently been addressed by Lin and Zeng \cite{LinZeng}, who reduced the problem to computing only the generalized kernel of $\mathcal L$.  More precisely, they show that only a finite number of iterates $\mathcal L^k$, $k\geq 1$ must be considered.

Our emphasis in this work is to demonstrate how a modified approach to modulation analysis allows for the treatment of more singular nonlinearities.  In particular, we will import existing results by citation when possible and focus on presenting in full detail only those parts of the analysis that are genuinely new.  With that in mind, the paper is organized as follows: In Section 2, we recall known results about the ground state and carry out the analysis of the linearized operator, which culminates in the modulation theory developed in Section 3. In Section 4, we demonstrate the exponential convergence of nonscattering constrained solutions to the ground state as $t\to\infty$ and briefly discuss the proof of the main theorems above.  

\section{Spectral Analysis}

\subsection{The ground state and the linearized operator}


We will often work with solutions which are close (in $\dot H^1(\R^3)$) to $W$. It is then natural to write $u(t,x) = W(x) + v(t,x)$, so that $v = v_1+iv_2$ solves the equation
\begin{equation}\label{linearized_1}
	i \partial_t v + \Delta v + K(v) = iR(v)
\end{equation}
where
\begin{equation}
	K(v) = |x|^{-1}W^2(3v_1 + iv_2), \qtq{and} R(v) = i|x|^{-1}W^3G(W^{-1}v),
\end{equation}
with $G(z)= |1+z|^2(1+z)-1-2z-\overline z$ being the sum of the quadratic and cubic terms of the polynomial $|1+z|^2(1+z)$.  In particular we have $G(0) = G_z(0) = G_{\bar z}(0) = 0$.

Separating the linear part of \eqref{linearized_1} into its real and imaginary parts, one has
\begin{equation}\label{linearized_2}
		i \partial_t v_1 - \partial_t v_2 +L_+ v_1-iL_-v_2 = iR(v),
\end{equation}
where
\begin{equation}
	L_+ = -\Delta - 3|x|^{-1}W^2\qtq{and}
	L_- = -\Delta - |x|^{-1}W^2.
\end{equation}

Identifying the complex number $a + bi \in \C$ with $\begin{bmatrix*}a \\b \end{bmatrix*}\in \R^2$, one can define the operator $L : \dot H^{1} \to \dot H^{-1}$ via
\begin{equation}\label{L-def}
Lv = L_+ v_1+ iL_- v_2 =  \begin{bmatrix*}L_+ v_1\\ L_- v_2 \end{bmatrix*},
\end{equation}
whose corresponding quadratic form is the \textit{linearized energy}
\begin{equation}\label{Q-def}
	Q(v) = \tfrac12\la Lv, v\raa = \tfrac12 	\la L_+ v_1,v_1\raa  +  \tfrac12 \la L_- v_2,v_2\raa.
\end{equation}
This can be polarized to a bilinear form
\begin{equation}\label{B-def}
	B(v,w) = \tfrac12 \la L_+ v_1,w_1\raa  +  \tfrac12 \la L_- v_2,w_2\raa .
\end{equation}
Note that $Q$ and $B$ can be uniquely extended to bounded operators in $\dot H^1(\R^3)$ and $\dot H^1(\R^3)\times \dot H^1(\R^3)$, respectively.

We also consider the densely defined, closed operator $\mathcal L : \text{dom}(\mathcal L) \subset \dot H^{1} \to \dot H^1$
\begin{equation}\label{calL-def}
	\mathcal{L}v  = iLv= \begin{bmatrix*}-L_-v_2 \\ L_+ v_1\end{bmatrix*} 
\end{equation}
which can also be seen, in matrix form, as the composition of the symplectic operator $J : \text{dom}(J) \subset \dot H^{-1} \to \dot H^1$,
\begin{equation}
	J = \begin{bmatrix*}0 & -I\\I & 0 \end{bmatrix*} 
\end{equation}
and $L$.

The operator $\mathcal{L}$ represents the linear intertwining between the real and imaginary parts of \eqref{linearized_2}, as 
\begin{equation}
	\partial_t v + \mathcal{L} v  = R(v).
\end{equation}%

Another important player in the analysis is the generator of scaling
\begin{equation}\label{Lambda-def}
	\Lambda f = \tfrac12 f + x \cdot \nabla f = \tfrac{d}{d\lambda} \bigl[\lambda^{\frac12} f(\lambda x)\bigr]\big|_{ \lambda=1 },
\end{equation}
which is skew-symmetric in $\dot H^1(\mathbb{R}^3)$. Note that $\Lambda W$ and $W$ are orthogonal functions in $\dot H^1(\mathbb{R}^3)$.

\subsection{Coercivity in {$\dot H^1$}}


In this section, we prove the following


\begin{prop}\label{P:orthogonality} For all $f \in \dot{H}^1(\mathbb{R}^3,\mathbb{R})$,

\begin{enumerate}[label=\upshape(\roman*),ref=\theprop (\roman*)]
	\item\label{P:orthogonality_item:1} $ \la L_+ f,f \ra_{L^2} \geq c \|f\|_{\dot H^1}^2 - O\left[ \la  f,  W\ra^2_{\dot H^{1}} + \la f,  \Lambda W\ra^2_{\dot H^{1}} \right]$.
	\item\label{P:orthogonality_item:2} $\la L_- f,f\ra_{L^2} \geq c \|f\|_{\dot H^1}^2 - O\left[ \la f,  W\ra^2_{\dot H^{1}} \right]$
\end{enumerate}

\begin{re} The expression $\langle L_+ f, f\rangle_{L^2}$ is a slight abuse of notation, but it has a clear interpretation, namely:
\[
\la L_+ f,f \ra_{L^2} := \int |\nabla f|^2\,dx - 3\int |x|^{-1}W^2f^2\,dx.
\]
\end{re}



\begin{re}\label{OrtCon-v}


In the context of the energy-critical equation \eqref{NLS}, this result implies \cite[Proposition 2.21]{Liu2024arxiv}, namely: there exists $c>0$ such that if $v \in \dot H^1(\R^3,\C)$ and $v = v_1 + i v_2$, we have
	\begin{equation}
		\la L_+ v_1,v_1\ra_{L^2} +  \la L_- v_2,v_2 \ra_{L^2} \geq c \|v\|_{\dot H^1}^2,
	\end{equation}
provided $\la  v_1,   W\ra_{\dot H^{1}} = \la v_1,  \Lambda W\ra_{\dot H^{1}} = \la v_2,  W\ra_{\dot H^{1}} =  0$. 
\end{re}
\end{prop}

\begin{proof}[Proof of Proposition~\ref{P:orthogonality}] It is immediate to check that $L_- W = L_+ \Lambda W = 0$. Moreover, $L_+W 
= -2|x|^{-1}W^3 = 2 \Delta W$. That means
\begin{equation}\label{laL+W}
\la L_+W,W\ra_{L^2}=-2\|W\|_{\dot H^1}^2<0.
\end{equation}

We now recall that $W$ is a minimizer of the Weinstein functional
\begin{equation}
    J(h) = \frac{\displaystyle\int |\nabla h|^2\,dx}{\left(\displaystyle\int |x|^{-1}|h|^4\,dx\right)^{\frac{1}{2}}},
\end{equation}

on $\dot{H}^1(\R^3,\C)$.  Thus, writing $h = f + i g$, using  the condition $\frac{d^2}{d\epsilon^2}J(W+\epsilon h)\big|_{\epsilon=0}\geq 0$ {and the Pohozaev identity
$$
\int |\nabla W|^2\,dx =\int |x|^{-1}|W|^4\,dx=\tfrac{8\pi}{3},
$$
}
we obtain
\begin{equation}\label{eq:non_negativity_Q}
    \la L_+f,f\ra_{L^2}+\la L_-g,g\ra_{L^2} \geq -\frac{2}{\displaystyle\int |\nabla W|^2\,dx}\la f,  W\ra_{\dot H^{1}}^2.
\end{equation}

Thus we see that $L_-$ is non-negative and that $L_+$ is non-negative if $\la f,  W\ra_{\dot H^{1}} = 0$. 

We now focus on $L_+$. We begin by observing that
\begin{equation}\label{L+K}
    |\nabla|^{-1}L_+|\nabla|^{-1} = I - 3|\nabla|^{-1}|x|^{-1}W^2|\nabla|^{-1} =: I - K,
\end{equation}
where $|\nabla|^{-1}$ denotes the Fourier multiplier operator with symbol $|\xi|^{-1}$.

We claim that the operator $K:L^2\to L^2$ is compact. It is enough to prove that $|x|^{-1}W^2 |\nabla|^{-1}: L^2 \to \dot{H}^{-1}$ is compact. To this end, note that for $f \in L^2(\R^3,\R)$ we have
\begin{equation}
    \||x|^{-1}W^2 |\nabla|^{-1}f\|_{L^\frac{6}{5}} \leq \||x|^{-1}W^2\|_{L^\frac{3}{2}}\||\nabla|^{-1}f\|_{L^6} \lesssim 
    \|f\|_{L^2},
\end{equation}
\begin{align}
	\||\nabla|^{\frac{1}{2}}(|x|^{-1}W^2 |\nabla|^{-1}f)\|_{L^\frac{6}{5}} &\lesssim 
	\||\nabla|^{\frac{1}{2}}(|x|^{-1}W^2)\|_{L^\frac{3}{2}}\||\nabla|^{-1}f\|_{L^6} + \||x|^{-1}W^2\|_{L^2}\||\nabla|^{-\frac{1}{2}}f\|_{L^3} \\
	&\leq 
	\|\nabla(|x|^{-1}W^2)\|_{L^\frac{6}{5}} \|f\|_{L^2} + \||x|^{-1}W^2\|_{L^2}\|f\|_{L^2} \\
	&\lesssim \|f\|_{L^2},
\end{align}
and
\begin{equation}
    \||x|^{-1}W^2|\nabla|^{-1}f\|_{L^{\frac{6}{5}}_{\{|x|\geq R\}}}\leq \tfrac{1}{R}\|W^2f\|_{L^\frac{6}{5}} \leq \tfrac{1}{R}\|W\|_{L^3}^2\||\nabla|^{-1}f\|_{L^6} \lesssim \tfrac{1}{R}\|f\|_{L^2}.
\end{equation}
Therefore, by the continuous embedding $L^{\frac{6}{5}} (\mathbb{R}^3) \hookrightarrow \dot{H}^{-1}(\mathbb{R}^3)$ and Rellich--Kondrachov Theorem, the compactness follows.

We then conclude that the eigenvalues of $I-K$ are discrete, and can only accumulate at $1$. Therefore, in the interval $(-\infty,\tfrac{1}{2}]$ the operator $I-K$ has at most a finite number of eigenvalues (counting multiplicity), say  $\{\lambda_i\}_{i=1}^{N}$, which we assume are ordered in a non-decreasing way.

By definition of the operator $K$, we have for $v\in \dot H^1(\mathbb{R}^3)$ that
$$
\la L_+v,v\ra_{L^2}=\la (I-K)w,w\ra_{L^2},
$$
where $w=|\nabla|v\in L^2(\mathbb{R}^3)$. From \eqref{laL+W} and \eqref{eq:non_negativity_Q}, we then obtain that 
\[
\la (I-K)|\nabla|W,|\nabla|W\ra_{L^2}<0
\]
 and that $I-K$ is non-negative if $\la w, |\nabla|W\ra_{L^2} = 0$, which implies $\lambda_1<0$ and $\lambda_2=0$. 

Since \eqref{L+K} ensures that $\ker L_+$ and $\ker(I-K)$ have the same dimension, it remains to show that $\ker L_+$ is spanned by $\Lambda W$. To this end, let $f$ be a solution to $L_+f=0$ and consider the spherical harmonics expansion
\begin{equation}
    f = \sum_{j=0}^{\infty}\sum_{m=-j}^{j} f_j(r) Y_j^m(\theta),
\end{equation}
where $\{Y_j^m\}_{j,m}$ is an orthonormal basis of $L^2(\mathbb{S}^2)$ satisfying
\begin{equation}
    -\Delta_{\mathbb{S}^2} Y_j^m = \mu_j Y_j^m\qtq{for all}j.
\end{equation}

We know that $\mu_j = j(j+1)$, with multiplicity $2j+1$. Recalling that $\Delta=\partial_{rr}+\frac{2}{r}\partial_r+\frac{1}{r^2}\Delta_{\mathbb{S}^2}$, we write $L_+f = 0$ in the spherical harmonics expansion as
\begin{equation}
    L_+ f = -\sum_{j=0}^{\infty} \sum_{m=-j}^{j}\left[(\partial_{rr}+\tfrac{2}{r}\partial_r-\tfrac{\mu_j}{r^2}+\tfrac{3}{r} W^2)f_j(r)\right] Y^m_j(\theta) = 0.
\end{equation}

Our goal is to show that $f_0\in \textrm{span}\{\Lambda W\}$ and $f_j=0$ for $j\geq 1$. We split the analysis in three cases.

\textbf{Case 1:} $j=0$.

In this case, $\mu_0 = 0$ and $Y_0^0 = 1$. Therefore:
\begin{equation}
    (\partial_{rr}+\tfrac{2}{r}\partial_r+\tfrac{3}{r} W^2)f_0(r)=0.
\end{equation}

We already know that $\Lambda W$ is a solution to this equation. Suppose now that $f_0$ were linearly independent of $\Lambda W$. By Abel's theorem, the Wronskian between $\Lambda W$ and $f$ must satisfy
\begin{equation}
     \Lambda W \partial_r f_0 - f_0 \partial_r \Lambda W = \tfrac{C}{r^2}.
\end{equation}
Since $r =2$ is the only zero of $\Lambda W(r) =\frac{2-r}{(r+2)^2}$, we can divide by $(\Lambda W)^2$ in a neighborhood of $r=0$ to obtain
{
\begin{equation}
    \partial_r \left(\frac{f_0}{\Lambda W} \right) = \frac{C(r+2)^4}{r^2(2-r)^2}.
\end{equation}
Integrating from $r$ to $1$, we get
\begin{equation}
    |f_0(r)| = \frac{2-r}{(r+2)^2}\left|\frac{f_0(1)}{9}-C \left(\frac{4}{r} -r - 12 \ln(r) +61-\frac{64}{2-r}\right)\right| \gtrsim \frac{1}{r}\qtq{as} r \to 0^+.
\end{equation}
}

Therefore, $f_0$ cannot belong to $L^6(\mathbb{R}^3)$, which implies $f_0 \notin \dot{H}^1(\mathbb{R}^3)$. Hence, the only radial solution is $\Lambda W$.

\textbf{Case 2:} $j=1$.

We have that $\mu_1 = 2$, with multiplicity 3. If there is a radial function $f_1$ such that $L_+(f_1Y_1^m) = 0$, for some $m \in \{-1,0,1\}$,  then
\begin{equation}
    (\partial_{rr}+\tfrac{2}{r}\partial_r-\tfrac{2}{r^2}+\tfrac{3}{r} W^2)f_1(r) = 0.
\end{equation}

This second-order equation has an explicit general solution given by
\begin{align}
    f_1(r) = c_1\frac{(r+2)^3}{r^2} + c_2 \frac{r^3+10r^2+40r}{(r+2)^2}.
\end{align}
It is immediate to check that no nonzero solution $f_1$ belongs to $\dot{H}^1(\R_+,r^2\,dr)$.

\textbf{Case 3:} $j>1$.

In this case, $\mu_j>2$ and we look for solutions to 
\begin{equation}
    (\partial_{rr}+\tfrac{2}{r}\partial_r-\tfrac{\mu_j}{r^2}+\tfrac{3}{r} W^2)f_j(r)=0.
\end{equation}
%
We first note that since $Y_1^0$ is orthogonal to the constant spherical harmonic $Y_0^0 \equiv 1$, we have 
\begin{equation}
    \int_{\mathbb{S}^2} Y_1^0 = 0.
\end{equation}
As $W$ is radial, this implies
\begin{equation}
    \la f_j Y_1^0,W\ra_{\dot{H}^1} = 0.
\end{equation}
Next, we compute
\begin{align}
    -\langle &(\partial_{rr}+\tfrac{2}{r}\partial_r-\tfrac{\mu_j}{r^2}+\tfrac{3}{r} W^2) f_j,f_j\rangle_{L^2(\R_+,r^2\,dr)} \\
    &= -\la(\partial_{rr}+\tfrac{2}{r}\partial_r-\tfrac{2}{r^2}+\tfrac{3}{r} W^2)f_j,f_j\ra_{\hspace{-1pt}L^2(\R_+,r^2\,dr)}\hspace{-2pt}+(\mu_j-2)\hspace*{-2pt}\int \tfrac{|f_j(|x|)|^2}{|x|^2}\,dx\\
    &>-\la\bigl[(\partial_{rr}+\tfrac{2}{r}\partial_r-\tfrac{2}{r^2}+\tfrac{3}{r} W^2)f_j\bigr]Y_1^0,f_jY_1^0\ra_{L^2(\mathbb S^2 \times \R_+, r^2\,d\sigma_{\mathbb S^2}dr)}\\
    &=-\la\bigl(\Delta+\tfrac{3}{|x|} W^2\bigr)(f_jY_1^0),f_jY_1^0\ra_{L^2(\R^3)}\\
    &=\la L_+(f_jY_1^0),f_jY_1^0\ra_{L^2(\R^3)}\geq 0.
\end{align}
The last inequality follows from the fact that $L_+$ is non-negative on directions which are orthogonal to $W$. Thus, there cannot exist any kernel function associated to $\mu_j$ if $j>1$.

Based on the analysis of the three cases above, it follows that $\ker L_+$ is one-dimensional and spanned by $\Lambda W$. Moreover, $\lambda_3>0$ and 
\begin{equation}
\la L_+ f,f\ra_{L^2} \geq \lambda_3 \|f\|_{\dot H^1}^2,\quad \mbox{if} \quad \la  f,   W\ra_{\dot H^{1}} = \la f,  \Lambda W\ra_{\dot H^{1}}  =  0.
	\end{equation}

For general $f \in \dot H^1 (\R^3, \R)$ we now write 
\begin{equation}
	\gamma_1 = \frac{\la f,  W\ra_{\dot H^{1}}}{\|\nabla W\|_{L^2}^2}\qtq{and} \gamma_2 = \frac{\la f, \Lambda W\ra_{\dot H^{1}}}{\|\nabla \Lambda W\|_{L^2}^2},
\end{equation}
so that $g:= f - \gamma_1 W - \gamma_2 \Lambda W\perp \text{span}\{W,\Lambda W\}$ in $\dot{H}^1$. We then have
\begin{equation}
	\|g\|_{\dot H^1}^2 \lesssim \la L_+g,g\raa = \la L_+f,f\raa-2\gamma_1 \la L_+f,W\raa + \gamma_1^2 \la L_+W,W\raa.
\end{equation}
Noting that $\|g\|_{\dot H^1}^2 = \|f\|_{\dot H^1}^2 - \gamma_1^2 \|W\|_{\dot H^1}^2 - \gamma_2^2\|\Lambda W\|_{\dot H^1}^2$ and $|\gamma_1\la L_+f,W\raa| \lesssim \epsilon \|f\|_{\dot H^1}^2 + C_{\epsilon}\gamma_1^2$, we conclude Proposition \ref{P:orthogonality_item:1}.

For the corresponding result for $L_-$, we note that,
\begin{equation}
	\la L_-f,f\raa  = \tfrac{2}{3} \|f\|_{\dot{H}^1}^2 + \tfrac{1}{3}\la L_+f,f\raa \geq  \tfrac{2}{3}\|f\|_{\dot{H}^1}^2
\end{equation}
if $\la f, W\rah=0$, since $L_+$ is non-negative on this subspace. The remainder of the proof follows analogously.\end{proof}

\subsection{Eigenvalues and generalized kernel}

We aim to show that $\mathcal{L}$ has exactly two real eigenvalues, which we denote by $\pm e_0$, with $e_0 >0$. To this end we are interested in solving
\begin{equation}\label{eq:eigenvalue_Y1_Y2}
	\begin{cases}
		L_+ Y_1 &= -e_0 Y_2,\\
		L_- Y_2 &= e_0 Y_1
	\end{cases}
\end{equation}
in $\dot H^1(\R^3, \R) \times \dot H^1(\R^3, \R)$. In \cite{DM_Thre}, the idea is to minimize the quadratic form generated by the operator $P = L_-^{\frac12}L_+L_-^{\frac12}$. In that context, $P$ is a relatively compact perturbation of $(-\Delta)^2$ in $H^4$, since the ground state is smooth. In the inhomogenous context, Liu--Yang--Zhang \cite{Liu2024arxiv} work in the regime $0<b<\min\{N/4,N-2\}$ and claim an equivalent result. However, observing that
$$
\|Pf\|^2_{L^2}=\int (L_-^{\frac12}f)L_+L_-L_+(L_-^{\frac12}{f}),
$$
we see that the operator $P$ might not even be bounded from $H^4$ to $L^2$, since $L_+L_-L_+$ behaves as singularly as $|x|^{-b-4}f$ around the origin.


We therefore take a different approach, in the spirit of Yang, Zeng and Zhang \cite{YZZ22}.  This approach relies in turn on a result of Lin and Zeng \cite{LinZeng}, which reduces the problem to precluding nonpositive directions of the quadratic form $Q$
restricted to the generalized kernel of 
$\mathcal{L}$ 
modulo $\ker L$. In particular, we prove the following: 
\begin{prop}\label{prop:generalized_kernel}
We have that $\ker \mathcal{L}^2 = \ker \mathcal{L} = \text{span}\{\Lambda W, iW\}$. Moreover, the number of negative directions of $Q$ in $\dot H^1(\R^3, \C)$ is $1$. 
\end{prop}
\begin{proof}
	Recall that the kernel of $\mathcal L$ in $\dot H^1(\R^3, \C)$ is spanned by $\Lambda W$ and $iW$. Now, if $\mathcal{L}^2v = 0$, then $\mathcal{L}v \in \ker \mathcal{L}$, which means $\mathcal{L}v = a \Lambda W + b i W$, for $a, b \in \R$. However, since $\mathcal{L}$ maps $\dot H^1(\R^3, \C)$ into $\dot H^{-1}(\R^3, \C)$, it is enough to show that $W, \,\Lambda W \notin \dot H^{-1}$ to conclude $a=b=0$. This can be seen by testing $W$ and $\Lambda W$ against the family of $\dot H^1$ norm-invariant rescalings $\{{n^{-1/2}}\phi(\cdot/n)\}_{n \in \mathbb Z_{>0}}$ of a bump function $\phi$ adapted to the unit ball centered at the origin.

	For the negative directions of $Q$, we recall that $Q(W) < 0$. Now, by \eqref{eq:non_negativity_Q}, we have that $Q(v) \geq 0$ if $\la v, W\rah = 0$, which guarantees that the number of negative directions does not exceed 1.
\end{proof}

With this result in hand, we now have grounds to invoke \cite[Theorem 2.2]{LinZeng} in order to conclude:
\begin{prop}
	The operator $\mathcal{L}$ defined in \eqref{calL-def}, generates a $C^0$ group $e^{t\mathcal{L}}$ of bounded linear operators in $\dot H^1(\R^3, \C)$, and there is a decomposition
	\begin{equation}
		\dot H^1(\R^3, \C) = E^u \oplus E^c \oplus E^s,
	\end{equation}
	into the \textit{unstable}, \textit{center} and \textit{stable} spaces of $\mathcal{L}$ satisfying:
	\begin{enumerate}
		\item $E^u$, $E^c$, $E^s \subset \text{dom}(\mathcal{L})$ are invariant under $e^{t\mathcal{L}}$,
 		\item $\dim E^u = \dim E^s = 1$,
		\item There exists $e_0>0$ such that
		\begin{align}
			\label{eq:e_0_E^u}e^{t\mathcal{L}}v &= e^{e_0 t} v \quad \forall v \in E^{u},\\
			\label{eq:e_0_E^s}e^{t\mathcal{L}}w &= e^{-e_0 t} w \quad \forall w \in E^{s}.
		\end{align}
	
	\end{enumerate}
\end{prop}
\begin{proof}
The decomposition of $\dot H^1$ and item \textit{1} follow directly from Theorem 2.2 in \cite{LinZeng}. Moreover, since $L$ has exactly one negative direction, \cite[Theorem 2.3]{LinZeng} states that
\begin{equation}
\dim E^u  + k^{\leq 0}_0 = 1,
\end{equation}
where, by equation (2.12) of \cite{LinZeng} $k^{\leq 0}_0$ is the number of nonpositive directions of the quadratic form  $Q$ restricted to $\ker_g \mathcal L/\ker L$, \textit{i.e.}, the quocient space between the generalized kernel of $\mathcal{L}$ and the kernel of $L$. By Proposition \ref{prop:generalized_kernel}, $\ker_g\mathcal L/\ker L = \{0\}$, which guarantees $k^{\leq 0}_0 = 0$. Therefore, $\dim E^u = 1$, which proves item \textit{2}. Next, equations \eqref{eq:e_0_E^u} and \eqref{eq:e_0_E^s} follow from the just proved items \textit{1} and \textit{2}, and from item \textit{6} of \cite[Theorem 2.1]{LinZeng}. This implies the existence of $e_0 \in \mathbb C$ with $\Re e_0>0$ and of a function $\mathcal{Y}_+ \in E^u$ such that $Y_1 := \Re \mathcal Y_+$ and $Y_2 :=\Im\mathcal Y_+$ satisfy \eqref{eq:eigenvalue_Y1_Y2}. Since $L_+$ and $L_-$ are real-valued operators, we conclude $e_0 \in \R$, which finishes the proof.
\end{proof}

\begin{re}
	As a consequence of this proposition, we can define $\mathcal Y_-  := \overline{\mathcal Y_{+}}$ so that $\mathcal{L} \mathcal Y_{\pm} = \pm e_0 \mathcal Y_{\pm}$. Using the system \eqref{eq:eigenvalue_Y1_Y2} and the weighted critical Sobolev inequality \eqref{weighted-sobolev}, it is then immediate to check that
	\begin{equation}
		e_0\|\mathcal{Y}_{+}\|^2_{L^2}  = 2\int|x|^{-1}W^2Y_1Y_2\,dx\leq 2\|W\|^2_{\dot H^1}\|Y_1\|_{\dot H^1}\|Y_2\|_{\dot H^1} < \infty.
	\end{equation}
Now, since $-\Delta Y_1 = -e_0 Y_2 + 3 |x|^{-1}W^2Y_1$ and $-\Delta Y_2 = e_0 Y_1 +  |x|^{-1}W^2Y_2$, the inequality
$$
\|x|^{-1}W^2Y_j\|_{L^2}\leq \|W\|^2_{L^{\infty}}\|\nabla Y_j\|_{L^2}
$$
implies $\mathcal{Y}_{\pm} \in H^2(\R^3)$ and Sobolev embedding gives $\mathcal{Y}_{\pm} \in L^\infty(\R^3)$. 

We further claim that 
\begin{equation}\label{Y-props}
\mathcal{Y}_\pm\in H^{3,p}(\R^3)\qtq{for all}1\leq p<\tfrac32,
\end{equation} so that (in particular) $(-\Delta)^{\frac32}\mathcal{Y}_{\pm}$ belongs to $L^{\frac65}(\R^3)\subset L^{\frac65,2}(\R^3)$ (as in \cite[Lemma 2.23(b)]{Liu2024arxiv}). 

To prove \eqref{Y-props}, first note 
$$
|\nabla \Delta Y_1|\lesssim  |\nabla Y_2| +  |x|^{-2}W^2|Y_1|+|x|^{-1}|W\nabla W| |Y_1|+|x|^{-1}W^2|\nabla Y_1|
$$
and observe the inequalities
\begin{equation}
	\||x|^{-2}W^2 Y_j\|_{L^p} \lesssim \||x|^{-1}W\|_{L^{2p}}^2  \|Y_j\|_{L^{\infty}} \lesssim \|\nabla  W\|_{L^{2p}} \|Y_j\|_{H^2} <\infty, \quad 1\leq p < 3/2,
\end{equation}
\begin{equation}
		\||x|^{-1}W^2 |\nabla Y_j|\|_{L^p} \lesssim \||x|^{-1}W^2\|_{L^{6p/(6-p)}}  \|\nabla Y_j\|_{L^6}  <\infty, \quad 1\leq p<2.
	\end{equation}
Furthermore, since $|W\nabla W| \in L^{\infty}$,
\begin{equation}
	\||x|^{-1}|W\nabla W| Y_j\|_{L^p} \lesssim \|W\nabla W\|_{L^{\infty}}  \||x|^{-1}Y_j\|_{L^p} \lesssim \|W\nabla W\|_{L^{\infty}} \|\nabla Y_j\|_{L^p} <\infty, \quad 1\leq p < 3.
\end{equation}
Combining the estimates above, we derive \eqref{Y-props}. 
\end{re}

Using a standard (co-)dimensional counting, we can then state coercivity under adapted orthogonality conditions related to the bilinear form \eqref{B-def}:

\begin{prop}[{\cite[Lemma 5.2]{DM_Thre}}, {\cite[Lemma 2.24]{Liu2024arxiv}}] If $v = v_1 + i v_2 \in \dot H^1(\R^3, \C)$ is such that $\la v_1, \Lambda W\rah = \la v_2, W \rah = B(\mathcal Y_+,v) = B(\mathcal Y_-,v)= 0$, then
	\begin{equation}
		Q(v) \gtrsim \|v\|_{\dot H^1}^2.
	\end{equation}
\end{prop}

\section{Modulation Theory}
With Proposition \ref{P:orthogonality} in hand, we are in a position to carry out an adapted modulation theory using a smoothed version of $W$. 

First, we define
\begin{equation}
	\delta(u) = \left|\|u\|_{\dot{H}^1}^2 - \|W\|_{\dot{H}^1}^2 \right|,
\end{equation}
which is a quantity whose smallness determines that we are in the modulation regime. We prove the modulation results in two steps. 

We then introduce the Littlewood--Paley projection operators.  We fix a radially decreasing function $\phi \in C^\infty_c(\R^3)$ satisfying $\phi(x) = 1$ if $|x|\leq 1$ and $\phi(x) = 0$ if $|x| \geq 2$.  Given $M\in 2^{\mathbb{Z}}$, we define the Fourier multiplier operator $P_{\leq M}$ via
\begin{equation}
	\widehat{P_{\leq M} f}(\xi) = \phi(\tfrac{\xi}{M})\hat f(\xi),
\end{equation}
where $\hat{\ }$ denotes Fourier transform.  Likewise, $P_{>M} f := f - P_{\leq M}f$.  

Note that $P_{\leq M} W$ remains real-valued for any choice of $M$.  The parameter $M$ will ultimately be fixed, chosen sufficiently large depending only on $W$.

In the following proposition, recall that $\Lambda$ is the generator of scaling (defined in \eqref{Lambda-def}).  We also continue to use the notation $v=v_1+iv_2$ for complex-valued functions.


		\begin{prop}[Modulation I]\label{P:new_modulation} Let $u:I\times\mathbb R^3\to \mathbb C$ be a solution to \eqref{NLS} satisfying $E(u_0)=E(W)$.  Let 
		\begin{equation}\label{I_0}
		I_0=\{t\in I: \delta(u(t))<\delta_0\}.
		\end{equation}
		If $M>0$ is sufficiently large and $\delta_0$ is sufficiently small (both depending only on $W$), then there exist $\theta:I_0\to\R/{2\pi \mathbb Z}
		$, $\lambda: I_0 \to (0,+\infty)$ and $v:I_0\to \dot H^1$ such that
		\begin{equation}\label{mod-def-v}
u(t,x) = e^{-i\theta(t)}\lambda(t)^{-1/2}[W(\lambda(t)^{-1} x)+v(t,\lambda(t)^{-1} x)],
		\end{equation}    
where the parameters $\lambda$ and $\theta$ are chosen in order to impose the orthogonality conditions
\begin{equation}\label{Orth-v1v2}
\la v_1,  P_{\leq M}\Lambda W\rah = \la  v_2 , P_{\leq M}W\rah = 0
\end{equation}
for all $t\in I_0$. Moreover, we have
\begin{equation}\label{mod-bds-pt1}
\|v(t)\|_{\dot H^1} \sim_{\delta_0} \delta(u) \qtq{ and }
|\dot \theta| + \big|\tfrac{\dot \lambda}{\lambda}\big| \lesssim_{\delta_0} M^2\lambda^{-2}\delta(u)
\end{equation}	
uniformly for $t\in I_0$.
\end{prop}
	\begin{proof}
Our approach builds on the strategies appearing in \cite[Lemma 3.7]{DM_Thre} and \cite[Lemma 4.4]{Liu2024arxiv}).  

We let $u$ be a solution to \eqref{NLS} satisfying $E[u_0] = E[W]$. The existence of functions $\theta_0: I_0 \to\R/{2\pi \mathbb Z} $ and $\lambda_0: I_0 \to (0,+\infty)$ such that 
\begin{equation}
\|u(t,x) - e^{-i\theta_0(t)}\lambda_0^{-1/2}(t)W(\lambda_0^{-1}(t)x)\|_{\dot{H}^1_x}= o_{\delta_0}(1)
\end{equation}
follows from the variational characterization of $W$.
		
Let $M>0$ to be determined below.  We will now apply the Implicit Function Theorem to the functional
$$
J(\theta, \lambda, h) := \left(\la  e^{i\theta}\lambda^{1/2}h(\lambda\; \cdot),iP_{\leq M}W\ra_{\dot H^1},
\la e^{i\theta}\lambda^{1/2}h(\lambda\; \cdot),\Pi (P_{\leq M}\Lambda W)\ra_{\dot H^1}\right),
$$ 
where 
\[
\Pi(f):=f- \frac{\la P_{\leq M}W ,f \ra_{\dot H^1}}{\|P_{\leq M}W\|^2_{\dot H^1}} P_{\leq M}W
\]
is the projection to the orthogonal complement of $P_{\leq M} W$.
		
		
To this end, we note that $J(0,1,P_{\leq M}W)=(0,0)$ and that
\begin{equation}
		\frac{\partial J}{\partial(\theta, \lambda)}(0,1,P_{\leq M}W) =\begin{bmatrix}{}
			\|P_{\leq M}W\|_{\dot H^1}^2 & 0\\
			0& \la \Lambda(P_{\leq M}W), \Pi(P_{\leq M}\Lambda W)\ra_{\dot H^1} \\
		\end{bmatrix}.
	\end{equation}	
By definition,
$$
\Pi(P_{\leq M}\Lambda W)=P_{\leq M}\Lambda W- \frac{\la P_{\leq M}W ,P_{\leq M}\Lambda W \ra_{\dot H^1}}{\|P_{\leq M}W\|^2_{\dot H^1}} P_{\leq M}(W)
$$
Noting that $P_{\leq M} f= f+o_M(1)$ in $\dot H^1$ as $M\to\infty$ for any $f\in \dot H^1$ and recalling $W\perp\Lambda W$, we derive that 
$$
\Pi(P_{\leq M}\Lambda W)= \Lambda W + o_M(1) \qtq{in} \dot H^1\qtq{as}\quad M\to \infty.
$$
On the other hand, we write
$$
\Lambda(P_{\leq M}W)=P_{\leq M}\Lambda W+[\Lambda; P_{\leq M}] W,
$$
where the commutator can be written on the Fourier side as
$$
\mathcal F\left({[\Lambda; P_{\leq M}]W}\right)(\xi)=m(\tfrac{\xi}{M})\widehat{W}(\xi),
$$
where $m(\eta):=\eta\cdot\nabla\phi(\eta)\in C_c^{\infty}$ is supported in the annulus $1\leq |\eta|\leq 2$. Therefore
$$
\|[\Lambda; P_{\leq M}] W\|_{\dot H^1}^2\lesssim \int_{M\leq |\eta|\leq 2M} |\xi|^2|\widehat{W}(\xi)|^2d\xi\leq \|P_{> M/2}W\|_{\dot H^1}^2=o_M(1), \quad \text{as}\quad M\to \infty.
$$
Thus, choosing $M= M(W)$ sufficiently large, we may ensure that 
\begin{equation}\label{detnotzero}
\left|\frac{\partial J}{\partial(\theta, \lambda)}(0,1,P_{\leq M}W)\right| \gtrsim_W 1
\end{equation}
and consequently
\begin{equation}
		\frac{\partial J}{\partial(\theta, \lambda)}(0,1,P_{\leq M}(W)) \in GL(2,\R).
	\end{equation}

Therefore, by the Implicit Function Theorem, there exists a ball $B = B(P_{\leq M}W,\epsilon_1)$ in $\dot H^1$, with $0<\epsilon_1\ll_W 1$ (independent of $M$), and functions $\theta_1: B \to \R/{2\pi \mathbb Z}$, $\lambda_1: B \to (0,\infty)$ satisfying \[
|\theta_1(h)| + |\lambda_1(h)-1| \ll_{W} 1\qtq{and}J(\theta_1(h),\lambda_1(h),h) = 0 \qtq{for all}h\in B.
\]

We now claim that $e^{i\theta_0(t)}\lambda_0^{1/2}(t)u(t,\lambda_0(t)\cdot) \in B$ for all $t \in I_0$ provided $\delta_0\ll_W 1$ and $M\gg_W 1$. Indeed, we have
\begin{align}
\|e^{i\theta_0(t)}\lambda_0^{1/2}(t)u(t,\lambda_0(t)\cdot) - P_{\leq M}W\|_{\dot H^1} &\leq \|e^{i\theta_0(t)}\lambda_0^{1/2}(t)u(t,\lambda_0(t)\cdot) - W\|_{\dot H^1} + \|W - P_{\leq M}W\|_{\dot H^1}\\
&= o_{\delta_0}(1) + \|P_{>M} W\|_{\dot H^1}< \epsilon_1
\end{align}
		if $\delta_0 \ll_W 1$ and $M \gg_W 1$.  This allows us to define
		\begin{equation}
		\theta(t) = \theta_0(t) +  \theta_1(e^{i\theta_0(t)}\lambda_0^{1/2}(t)u(t,\lambda_0(t)\cdot))
		\end{equation}
		and
		\begin{equation}
			\lambda(t) = \lambda_0(t) \lambda_1(e^{i\theta_0(t)}\lambda_0^{1/2}(t)u(t,\lambda_0(t)\cdot)), 
		\end{equation}
such that $J(\theta(t),\lambda(t),u(t)) = 0$ for all $t \in I_0$.  

It remains to verify the differentiability of $\theta$ and $\lambda$. To this end, let us introduce 
\begin{equation}
F(t,\omega,\mu) = \left(\la  e^{i\omega}\mu^{1/2}u(t,\mu\; \cdot),iP_{\leq M}W\ra_{\dot H^1},\la e^{i\omega}\mu^{1/2}u(t,\mu\; \cdot),\Pi (P_{\leq M}\Lambda W)\ra_{\dot H^1}\right).
\end{equation}
Let $t_0 \in I_0$ and write 
\[
v_0 =e^{i\theta(t_0)}\lambda(t_0)^{1/2}u(t,\lambda(t_0)\; \cdot)-P_{\leq M}W,
\]
so that $\|v_0\|_{\dot H^1}<\epsilon_1\ll_W 1$.  We then compute 

	\begin{align}
			\frac{\partial F}{\partial(\omega, \mu)} (t_0,\theta(t_0),\lambda(t_0))
			&=\begin{bmatrix}{}
				\|P_{\leq M}W\|_{\dot H^1}^2 & 0\\
			0& \lambda^{-1}(t_0)\la  \Lambda(P_{\leq M}W), \Pi(P_{\leq M}\Lambda W)\ra_{\dot H^1} \\
			\end{bmatrix}\\
			&\label{projection_terms_to_be_bounded}\quad\quad+\begin{bmatrix}{}
				\la iv_0,iP_{\leq M}W\ra_{\dot H^1}&  -\lambda^{-1}(t_0)\la v_0,i\Lambda P_{\leq M}W\ra_{\dot H^1}\\
				\la iv_0,\Pi (P_{\leq M}\Lambda W)\ra_{\dot H^1}& -\lambda^{-1}(t_0)\la v_0,\Lambda\Pi (P_{\leq M}\Lambda W)\ra_{\dot H^1}
			\end{bmatrix}.
		\end{align}
Thus, arguing as in the proof of \eqref{detnotzero}, if $M$ is large enough (depending only on $W$), we have 
		\begin{equation}
			\left|\frac{\partial F}{\partial(\theta, \lambda)} (t_0,\theta(t_0),\lambda(t_0))\right| \gtrsim_W \lambda^{-1}(t_0)>0.
		\end{equation}
		We stress that all quantities in \eqref{projection_terms_to_be_bounded} are finite and uniformly bounded in $M$, as
		$W$, $\Lambda W$ and $\Lambda^2 W$ belong to $\dot H^1$ and $\|[P_{\leq M};\Lambda]\|_{\dot H^1 \to \dot H^1} \lesssim 1$.

Moreover, as $\partial_t u \in C^0_t \dot H^{-1}_x$ and $P_{\leq M} W$, $\Pi P_{\leq M}\Lambda W \in \dot H^1$, we have that $F$ is $C^1$. Thus, another application of the Implicit Function Theorem implies that there exist $C^1$ functions $(\tilde \theta_{t_0},\tilde \lambda_{t_0})$, defined on a neighborhood $U_{t_0}$ of $t_0$, such that
\begin{equation}
F(t,\omega,\mu)=0 \qtq{if and only if} (\omega,\mu)=(\tilde\theta_{t_0}(t),\tilde\lambda_{t_0}(t))
\end{equation}
for all $t\in U_{t_0}$.  However, by construction we have 
\[
F(t,\theta(t),\lambda(t)) \equiv J(\theta(t),\lambda(t),u(t))=0,
\]
so  that $\theta \equiv \tilde \theta_{t_0}$ and $\lambda \equiv \tilde \lambda_{t_0}$ on $U_{t_0}$. In particular, $\theta,\lambda$ are $C^1$ on a neighborhood of $t_0$.  As $t_0\in I_0$ was arbitrary, we obtain $\theta,\lambda\in C^1$. 

If we now define
\begin{equation}
v(t,x) =e^{i\theta(t)}\lambda(t)^{1/2}u(t,\lambda(t)\, x) - P_{\leq M} W(x),
\end{equation}
then the orthogonality conditions \eqref{Orth-v1v2} hold by construction.
		
We now turn to the bounds on the derivatives of the modulation parameters and on $v$. We write
		\begin{equation}\label{Decomp-v1}
			v=\beta W+\widetilde{v}, \quad \mbox{with} \quad \beta:= \frac{\la  v_1, W\ra_{\dot H^{1}}}{\|W\|_{\dot H^1}^2} \in \R,
		\end{equation}	
so as to impose the orthogonality condition  $\la \widetilde{v}_1\,, W\ra_{\dot H^1}=0$.  In particular, we have
		\begin{equation}\label{Ort-vtildv}
			\|v\|^2_{\dot{H}^1}= \beta^2\|W\|^2_{\dot{H}^1}+\|\widetilde{v}\|^2_{\dot{H}^1}
		\end{equation}
		and $B(W, \widetilde{v})$=0 (cf. \eqref{B-def} above). Therefore,
		\begin{equation}\label{Q-tildev}
			Q(v)=Q(\widetilde{v})+\beta^2Q(W)
		\end{equation}	
(cf. \eqref{Q-def}). Moreover, from the variational characterization of $W$ and the scaling invariance of the energy we also have
		$$
		E[W]=E[u]=E[W+v]=E[W]+Q(v)+O(\|v\|^3_{\dot{H}^1}),
		$$
		which implies
		\begin{equation}\label{Q-cubicv}
			Q(v)=O(\|v\|^3_{\dot{H}^1}).
		\end{equation}	
		Since $Q(W)<0$ by direct computation, we may use Proposition \ref{P:orthogonality} and the relations \eqref{Q-tildev}--\eqref{Q-cubicv}, to deduce
		$$
		\|\widetilde{v}\|^2_{\dot{H}^1}\lesssim Q(\widetilde{v})=Q(v)+\beta^2|Q(W)|\lesssim \beta^2+\|v\|^3_{\dot{H}^1}
		$$
		and
		$$
		\beta^2\lesssim Q(\widetilde{v})+Q(v)\lesssim \|\widetilde{v}\|^2_{\dot{H}^1}+\|v\|^3_{\dot{H}^1}.
		$$
As $v$ is small in $\dot{H}^1$, the relation \eqref{Ort-vtildv} implies that $|\beta| \sim \|\widetilde{v}\|_{\dot H^1} \sim \|v\|_{\dot H^1}\ll 1$. To relate these quantities to $\delta(u)$, we use the definition of $\beta$ to write
		$$
		\delta(u)=\left|\|W+v\|_{\dot{H}^1}^2 - \|W\|_{\dot{H}^1}^2 \right|=\left|\| v\|_{\dot{H}^1}^2 +2\beta \| W\|_{\dot{H}^1}^2 \right|,
		$$
so that $|\beta|\sim \delta(u)$.

		One can now directly verify that $v$ satisfies the equation
		\begin{equation}
			i\partial_t v + \lambda^{-2}\Delta v = -\dot \theta [W+v] + i \tfrac{\dot \lambda}{\lambda} \Lambda[W+v] - \lambda^{-2}|x|^{-1}[|W+v|^2(W+v) - |W|^2W],
		\end{equation}
		or equivalently
		\begin{align}\label{eq:system_v1_v2}
			\begin{cases}
				\dot v_1 &= \lambda^{-2}L_-v_2- \dot \theta v_2 + \frac{\dot \lambda}{\lambda}\Lambda[W+v_1] + \lambda^{-2} \Im R(v),\\
				\dot v_2 &= - \lambda^{-2}L_+ v_1 + \dot \theta [W+v_1] +\frac{\dot \lambda}{\lambda}\Lambda v_2 - \lambda^{-2}\Re R(v),
			\end{cases}
		\end{align}
		where $R(v) := -|x|^{-1}[2|v|^2W + v^2 W +|v|^2 v]$.
		
Testing the system above against 
\[
0+iP_{\leq M}W = \begin{bmatrix*}[c]0\\P_{\leq M} W\end{bmatrix*}\qtq{and} P_{\leq M}\Lambda W + 0i = \begin{bmatrix*}[c]P_{\leq M}\Lambda W\\0\end{bmatrix*}
\]
in $\dot H^1$ and integrating by parts, we deduce the following linear system for the time derivatives of the modulation parameters $\dot{\theta}$ and $\dot{\lambda}/\lambda$:
\[
\mathcal{A}\begin{bmatrix}
			\dot{\theta} \\[4pt]
			\frac{\dot{\lambda}}{\lambda}
		\end{bmatrix}
		= \lambda^{-2}
		\begin{bmatrix}
			F_1 \\[4pt]
			F_2
		\end{bmatrix},
\]
where
\[ \mathcal{A}:=
		\begin{bmatrix}
			\la W, P_{\leq M} W\ra_{\dot H^1}+\la  v_1, P_{\leq M}W\ra_{\dot H^1} &-\la   v_2,  \Lambda P_{\leq M}W\rah \\[4pt]
			- \la v_2,P_{\leq M}\Lambda W\rah& \la \Lambda W, P_{\leq M }\Lambda W\rah-\la  v_1,\Lambda P_{\leq M}\Lambda W\rah
		\end{bmatrix}
		\]
and (recalling \eqref{B-def}) we have
\begin{align*}
&F_1 = B(v_1,\Delta  P_{\leq M}W)+ \la \Re R(v),\, \Delta P_{\leq M}W \ra_{L^2},\\
&F_2 = -B(iv_2, i\Delta P_{\leq M} \Lambda W) - \la \Im R(v),\, \Delta P_{\leq M} \Lambda W \ra_{L^2}.
\end{align*}
		
As $M \gg_W 1$ and $\delta_0 \ll_W 1$, we have that
\[
|\det\mathcal{A}| \gtrsim \|W\|_{\dot{H}^1}^2\|\Lambda W\|_{\dot{H}^1}^2>0
\]
uniformly for $t\in I_0$. Moreover, we have
		$$
		|B(v_1,P_{\leq M}\Delta W)|+ |B(iv_2,i\Delta P_{\leq M}\Lambda W) |\lesssim M^2\| v\|_{\dot{H}^1}.
		$$
		
		For the nonlinear terms we use $|R(v)|\lesssim |x|^{-1}(W|v|^2+|v|^3)$ to deduce from \eqref{weighted-sobolev} that
		\begin{align}
|\la R(v), \Delta P_{\leq M}  W \ra_{L^2}| &\lesssim \|\Delta P_{\leq M}  W\|_{\dot H^1}(\|v\|_{\dot H^1}^2\|W\|_{\dot H^1}+\|v\|_{\dot H^1}^3)\\
&\lesssim M^2 (\|v\|_{\dot H^1}^2 + \|v\|_{\dot H^1}^3	),
\end{align}
		and similarly
	$$
		|\la R(v), \Delta P_{\leq M} \Lambda W \ra_{L^2}|\lesssim M^2 (\|v\|_{\dot H^1}^2 + \|v\|_{\dot H^1}^3	).
		$$	
		Solving the linear system and using $\|v\|_{\dot H^1} \sim \delta(u)$, we finally obtain		
		\begin{equation}
			|\dot \theta(t)| \lesssim M^2\lambda^{-2}(t)\delta(u(t)) 
		\qtq{and}
			|\dot \lambda(t)| \lesssim M^2\lambda^{-1}(t)\delta(u(t)),
		\end{equation}
		for all $t \in I_0$. This completes the proof.\end{proof}

Note that in Proposition~\ref{P:new_modulation}, we did \emph{not} obtain any control over the term $|\dot\beta|$, where $\beta$ is the coefficient that was introduced in \eqref{Decomp-v1}.  In fact, this is precisely the quantity that one cannot estimate satisfactorily when the parameter $b$ in the nonlinearity is too large (as larger $b$ corresponds to a stronger singularity of the ground state at the origin).  In the prior related work \cite{Liu2024arxiv}, the authors estimated this quantity by taking an inner product (in $\dot{H}^1$) between the first equation in \eqref{eq:system_v1_v2} and $W$ in order to obtain the identity
$$
\dot{\beta}\|W\|_{\dot{H}^1}^2=-\lambda^{-2}(\la v_2,\, L_- (\Delta W) \ra_{L^2}+\Im \la R(v),\, \Delta W \ra_{L^2}).
$$
In order to estimate the final term on the right-hand side, the authors ultimately had to impose $b\in (0,\tfrac12)$ (thereby limiting the singularity of the ground state at the origin).  We refer the reader to \cite[Lemma 4.4, Claim B.2]{Liu2024arxiv} for further details. 



In order to deal with the singular nonlinearity $|x|^{-1}|u|^2 u$, we modify the orthogonality conditions to \eqref{Orth-v1v2} and introduce the corresponding decomposition (with $v=v(t)$ as in \eqref{mod-def-v}):
\begin{equation}\label{Decomp-v2}
v = \alpha P_{\leq M} W + g, \quad \mbox{with}\quad  \alpha = \frac{\la  v_1,P_{\leq M} W\rah}{\|P_{\leq M}W\|_{\dot H^1}^2}\in\R.
\end{equation}
This yields the orthogonality condition 
\[
\la  g_1, P_{\leq M}W\rah = 0.
\]

We now demonstrate that for $M=M(W)$ sufficiently large, we can obtain coercivity for the expression $\langle Lg,g\rangle$ uniformly over the modulation interval $I_0$. 

\begin{lemma}\label{L:Coercive-LM} We have \begin{equation}\label{Coercive-L}
\la Lg, g\raa \gtrsim \|g\|_{\dot H^1}^2 - o_M(1)\,|\alpha|^2
\end{equation}
as $M\to\infty$, uniformly for $t\in I_0$. 
\end{lemma} 

\begin{proof} We will use Proposition~\ref{P:orthogonality}. Using the orthogonality conditions, we first compute
\[
|\langle g_1, W\rangle_{\dot H^1}| = |\langle g_1, P_{>M} W\rangle_{\dot H^1}| \lesssim o_M(1)\,\|g\|_{\dot H^1}.
\]
Next, as $g_2=v_2$, 
\[
|\langle g_2,W\rangle_{\dot H^1}| = |\langle g_2,P_{>M}W\rangle_{\dot H^1}| \lesssim o_M(1)\,\|g\|_{\dot H^1}. 
\]
Finally, we use the definition of $g$ and the orthogonality conditions to observe that
\begin{align*}
\langle g_1,\Lambda W\rangle_{\dot H^1} & = \langle v_1,P_{>M}\Lambda W\rangle_{\dot H^1} - \alpha \langle P_{\leq M}W,\Lambda W\rangle_{\dot H^1} \\
& = \langle g_1,P_{>M}\Lambda W\rangle_{\dot H^1} - \alpha\langle P_{\leq M}W,P_{\leq M}\Lambda W\rangle_{\dot H^1}. 
\end{align*}
Recalling that $W\perp\Lambda W$, we obtain
\[
|\langle g_1,\Lambda W\rangle_{\dot H^1}| \lesssim o_M(1)\,\bigl[\|g\|_{\dot H^1}+|\alpha|\bigr]. 
\]
Proposition~\ref{P:orthogonality} therefore implies that
\[
\langle Lg,g\rangle_{L^2} \gtrsim \|g\|_{\dot H^1}^2 - o_M(1)\,\|g\|_{\dot H^1}^2 - o_M(1)\,|\alpha|^2,
\]
which yields the result. \end{proof}

%
%

In the next proposition we utilize the previous lemma and obtain control over the parameters $\|g\|_{\dot H^1}$, $|\alpha|$, and $|\dot{\alpha}|$ in the decomposition \eqref{Decomp-v2}.

\begin{prop}[Modulation II
]\label{P:new_modulation_alpha}
 For $M = M(W)\gg 1$, 
    \begin{equation}
        \|g\|_{\dot H^1} \sim_{\delta_0} |\alpha| \sim_{\delta_0} \delta, \qtq{ and } |\dot \alpha | \lesssim_{\delta_0} M^2\lambda^{-2} \delta
    \end{equation}
    uniformly for $t\in I_0$.
\end{prop}
\begin{proof}
From Proposition \ref{P:new_modulation} we already have that $\|v\|_{\dot H^1} \lesssim \delta(u)$ and the orthogonality in the decomposition \eqref{Decomp-v2} gives 
	\begin{equation}\label{eq:decomposition_v}
		\|g\|^2_{\dot H^1} +  \alpha^2 \|P_{\leq M}W\|^2_{\dot H^1} = \|v\|^2_{\dot H^1} \lesssim \delta^2(u).
	\end{equation}
We now expand
	\begin{equation}\label{eq:Qgalpha}
		Q(v) = Q(g) + 2\alpha B(g,P_{\leq M}W) + \alpha^2 Q(P_{\leq M}W).
	\end{equation}
Now note that $| Q(P_{\leq M}W)| \gtrsim1$ uniformly for large $M$. Therefore, the estimates \eqref{Q-cubicv} and \eqref{Coercive-L} together with equations \eqref{eq:decomposition_v} and \eqref{eq:Qgalpha} imply, 
	\begin{align}
		\|g\|_{\dot H^1}^2 &\lesssim \alpha^2 + |\alpha| \|g\|_{\dot H^1} + \alpha^3 + \|g\|_{\dot H^1}^3
	\end{align}
	and
	\begin{align}
		\alpha^2 &\lesssim \|g\|_{\dot H^1}^2 + |\alpha| \|g\|_{\dot H^1} + \alpha^3 + \|g\|_{\dot H^1}^3.
	\end{align}

Since $0<\delta_0\ll_W 1$, we may use relation \eqref{eq:decomposition_v} again to obtain
\begin{equation}
	\|g\|_{\dot H^1} \sim |\alpha| \sim \|v\|_{\dot H^1}\lesssim \delta(u).
\end{equation}

We now expand the definition of $\delta$ and utilize the orthogonality relations to write 
\begin{equation}
	\delta(u) = \left| \|W + v\|_{\dot H^1}^2 - \|W \|_{\dot H^1}^2\right| = \left|\|v\|_{\dot H^1}^2+2 \alpha \la W,P_{\leq M}W\rah + \la g_1, P_{>M}W\rah \right|.
\end{equation}
Noting that 
\[
\|v\|_{\dot H^1}^2 \lesssim \|v\|_{\dot H^1} |\alpha| \ll |\alpha|\qtq{and} |\langle g_1,P_{>M}W\rangle_{\dot H^1}| \ll \|g\|_{\dot H^1} \ll \alpha,
\]
we derive that $\delta(u)\sim |\alpha|$.  

We turn to the estimate for $|\dot \alpha|$.  We begin by testing the first equation in \eqref{eq:system_v1_v2} against $ P_{\leq M} W$ in $\dot H^1$ to obtain
\begin{align}
|\la  W+v_1\,,\, \Lambda P_{\leq M} W\rah|
&\lesssim \|\Lambda P_{\leq M} W\|_{\dot H^1}(\| W\|_{\dot H^1} +\|v_1\|_{\dot H^1})
\end{align}
and, using the weighted Sobolev embedding \eqref{weighted-sobolev},
\begin{align}
|\la R(v),\, \Delta W_{\leq M} \ra_{L^2}|&\lesssim \|\Delta P_{\leq M}  W\|_{\dot H^1}(\|v\|_{\dot H^1}^2\|W\|_{\dot H^1}+\|v\|_{\dot H^1}^3)\\
&\lesssim M^2 (\|v\|_{\dot H^1}^2 + \|v\|_{\dot H^1}^3	).
\end{align}

Thus, recalling the estimates already obtained in Proposition~\ref{P:new_modulation}, we obtain  
\begin{align}
	|\dot \alpha|\|W_{\leq M}\|_{\dot H^1}^2 &\lesssim \lambda^{-2}\|v_2\|_{\dot H^1} \|\Delta W_{\leq M}\|_{\dot H^1} + |\dot \theta|\,\|v_2\|_{\dot H^1}\|W_{\leq M}\|_{\dot H^1} \\
	&\quad +\bigl|\tfrac{\dot \lambda}{\lambda}\bigr|\left(\|\Lambda W\|_{\dot H^1} \|W_{\leq M}\|_{\dot H^1}+\|v\|_{\dot H^1}(\|\Lambda W_{\leq M}\|_{\dot H^1}+\|W_{\leq M}\|_{\dot H^1})\right)\\ 
	&\quad+\lambda^{-2}\left(\|v\|_{\dot H^1}^2(\|\Delta W_{\leq M}\|_{\dot H^1}\|W\|_{\dot H^1})+\|v\|_{\dot H^1}^3\|\Delta W_{\leq M}\|_{\dot H^1}\right)
	\\ &\lesssim M^2\lambda^{-2}\delta(u).
\end{align}
This finally implies
$$
|\dot \alpha(t)| \lesssim M^2 \lambda^{-2}(t)\delta(u(t)),
$$
for all $t \in I_0$, as desired. 
\end{proof}

\section{Convergence of nonscattering, constrained solutions}

With the new modulation analysis in place, the proofs of Theorems~\ref{thm:Qpm}~and~\ref{thm:threshold} follow along fairly well-known lines.  

In particular, the spectral analysis developed in Section 2 allows us to derive the existence result in Theorem \ref{thm:Qpm} in the same way as \cite{DM_Thre} and \cite{Liu2024arxiv}.  The proof of exponential convergence of unconstrained solutions to the ground state in one time direction also follows along standard lines.  We remark that the radial assumption is still required to obtain this result using current techniques.  The most subtle aspect of the analysis is the exponential convergence of nonscattering constrained solutions to the ground state in one time direction, and so we focus on presenting this case in the present section.  Once the convergence for both constrained/unconstrained solutions is established, the analysis and classification of all threshold solutions, as stated in Theorem \ref{thm:threshold}, also follows as in previous works.

Thus in this final section we will focusing on presenting a proof of the convergence of nonscattering constrained solutions, leveraging results from the previous section on modulation analysis.  We will utilize several existing results from references such as \cite{YZZ22, Liu2024arxiv, DM_Thre, GM21}, as these works consider closely-related problems and many results from these works can be imported directly in our setting. 

Let $u_0$ satisfy $E(u_0) = E(W)$ and $\|u_0\|_{\dot H^1} < \|W\|_{\dot H^1}$, and let $u:(T^-,T^+)\times\R^3\to\C$ denote the corresponding maximal-lifespan solution to \eqref{NLS}.  By the sharp weighted Sobolev embedding estimate, we have $\|u(t)\|_{\dot H^1} \leq \|W\|_{\dot H^1}$ for all $t \in (T^-,T^+)$, and in particular $|\delta(u(t))|\lesssim 1$. We suppose that $u$ does not scatter forward in time, i.e. $\|u\|_{L^{10}_{t,x}([0,T^+)\times \R^3)}=\infty$. 

Standards arguments using concentration-compactness imply that the solution $u \in C^0_tH^1_x([0,T^+)\times \R^3) \cap L^{10}_{t,x}([0,T^+)\times \R^3)$ is almost periodic modulo scaling.  In particular, there exist $N: [0,T^+)\to (0,\infty)$ and a function $C:(0,\infty) \to (0,\infty)$ such that\footnote{The presence of the factor $|x|^{-1}$ in the nonlinearity removes the need for a moving spatial center, as observed in \cite[Theorem 1.2]{GM21}, for example.}
\begin{equation}
\int_{|x|>C(\eta)/N(t)} |\nabla u(t,x)|^2 \, dx + \int_{|\xi|\geq C(\eta)N(t)} |\xi \hat u(t,\xi)|^2 \, d\xi  \leq \eta\qtq{for all}t\in [0,T^+).
\end{equation}
For details, see e.g. \cite[Theorem 1.16]{KV10}, \cite[Theorem 1.2]{GM21}, or \cite[Proposition 6.2]{Liu2024arxiv}.

Equivalently, we have that the set
\begin{equation}
	K = \bigl\{ \tfrac{1}{N(t)^{1/2}} u\bigl({t}, \tfrac{\cdot}{N(t)} \bigr), \quad t \in [0,T^+)  \bigr\}
\end{equation}
is precompact in $\dot H^1$. 

We can rule out the finite-time blowup scenario $T^+<\infty$ using the  reduced Duhamel formula from \cite[Proposition 5.23]{KV13} as in \cite[Proposition 4.2]{GM21}. Thus, we assume that $T^+=+\infty$ and that $u$ does not scatter forward in time. 

Following the proof of \cite[Lemma 9.3]{YZZ22}, for example, the scale function $N(t)$ can be taken to be differentiable almost everywhere and satisfy 
\begin{equation}\label{invlambda2}
	\left|\tfrac{1}{N^2(t_1)}-\tfrac{1}{N^2(t_2)}\right|\lesssim \int_{t_1}^{t_2}\delta(u(s))\,ds\qtq{for all}  t_2>t_1>0.
\end{equation}

The scaling compactness parameter $N(t)$ can also be related to the modulation parameter $\lambda(t)$ from Proposition \ref{P:new_modulation}, using the fact that the set $K$ defined above is precompact. Indeed, following  \cite[Lemma 6.3]{Liu2024arxiv}, there exists $C \in (0,\infty)$ such that
\begin{equation}\label{CompNlambda}
	\tfrac1C < N(t)\lambda(t) < C \,\,\, \mbox{for all} \,\,\, t \in I_0\,\,\, \mbox{given in \eqref{I_0}}.
\end{equation}


Our first task is to conclude that $\delta(u(t)) \to 0$ as $t \to \infty$. To this end, we have the following estimate, based on combining the virial identity with the modulation analysis:

\begin{prop}[Modulated virial, {\cite[Lemma 6.5]{Liu2024arxiv}}]\label{ModVir} A solution $u$ that is precompact modulo scaling on $[0,\infty)$ satisfies\begin{equation}\label{eq:modulated_virial}
    \int_{t_1}^{t_2} \delta(u(s)) \,ds \lesssim \sup_{t \in [t_1,t_2]}\tfrac{1}{N(t)^2} [\delta(t_1) + \delta(t_2)]\qtq{for all} t_2>t_1\geq 0.
\end{equation}
\end{prop}

From the previous proposition, the fact that $\delta(u(t))\lesssim 1$ and the limit ${tN^2(t)} \to +\infty$ (a consequence of the local theory and compactness modulo symmetries, cf. \cite[proof of Lemma~3.3]{DM_Thre}), we have that
\begin{equation}
	\inf_{t \in [t_1, t_2]} \delta(u(t)) \lesssim \bigl(\sup_{t \in [t_1, t_2]} \tfrac{1}{tN^2(t)}\bigr) \frac{t_2}{t_2-t_1}\qtq{for} t_2>t_1\geq 0.
	\end{equation} 
This guarantees 
\[
\liminf_{t \to +\infty} \delta(u(t)) \leq \limsup_{t \to +\infty} \tfrac{1}{tN^2(t)}= 0.
\]
Therefore, there exists a sequence $t_n \to +\infty$ such that $\delta(u(t_n)) \to 0$. 

The existence of such a sequence, the relation \eqref{invlambda2}, and the modulated virial in Proposition~\ref{ModVir} allow us to show that the compactness scaling parameter is uniformly bounded above and away from zero. We begin by proving the lower bound.

\begin{prop}\label{infNt}
	We have that
	$$
	\inf_{t \geq 0} N(t)>0.
	$$
\end{prop}
\begin{proof}
Let $t_n \to \infty$ such that $\delta(u(t_n)) \to 0$. For $t_n\leq t\leq t_m$ we have, from the relation \eqref{invlambda2} and Proposition \ref{ModVir}, that
\begin{equation}
	\bigl|\tfrac{1}{N^2(t)}  - \tfrac{1}{N^2(t_n)}\bigr| + \bigl|\tfrac{1}{N^2(t_m)}  - \tfrac{1}{N^2(t)}\bigr| \lesssim \int_{t_{n}}^{t_m} \delta(u(s))\, ds \lesssim \sup_{\tau \in [t_{n},t_m]}\tfrac{1}{N(\tau)^2} (\delta(u(t_{n})) + \delta(u(t_m))),
\end{equation}
which implies
\begin{equation}\label{Nmn}
	\bigl(\sup_{t \in [t_{n},t_m]}\tfrac{1}{N^2(t)}\bigr) \left(1-o_{{n}}(1)-o_{m}(1) \right)  \leq \min\bigl\{\tfrac{1}{N^2(t_m)}, \tfrac{1}{N^2(t_n)} \bigr\}.
\end{equation}
Taking $n$ sufficiently large and letting $m \to +\infty$, we deduce that
$$
\sup_{t \in [t_{n},\infty)}\tfrac{1}{N^2(t)}(1-o_n(1))\lesssim \tfrac{1}{N^2(t_n)},
$$
proving the desired estimate.\end{proof}

Note that Propositions \ref{eq:modulated_virial} and \ref{infNt} applied to the sequence $t_n \to \infty$ such that $\delta(u(t_n)) \to 0$ show that 
$$\displaystyle\int_t^{\infty} \delta(u(s)) \, ds \lesssim \delta(u(t)),
$$ 
which, together with Gronwall’s inequality, implies 
\begin{equation}\label{Gron}
\displaystyle\int_t^{\infty} \delta(u(s)) \, ds \lesssim e^{-ct}\,\,\,\mbox{for some}\,\,\,c >0.
\end{equation}

We now turn to the proof of the upper bound.
\begin{prop}\label{supNt}
	We have that
	$$
	\sup_{t \geq 0} N(t)<\infty.
	$$
\end{prop}
\begin{proof}
From relations \eqref{invlambda2} and \eqref{Gron} we deduce that the limit  $\lim_{t\to \infty}[N(t)]^{-2}$ exists. It suffices to show that this limit is positive. If not, then using inequality \eqref{Nmn}, we take $n$ sufficiently large and let $m \to \infty$ to obtain
$$
\frac{1}{N^2(t)}=0 \,\,\,\mbox{for all}\,\,\, t\geq t_n,
$$
which yields a contradiction. 
\end{proof}
%
%


Now, for any $t$ such that $(t_n,t)\subset I_0$ (the modulation interval), we have from the relation \eqref{CompNlambda} and Propositions~\ref{infNt}~and~\ref{supNt} that the modulation scaling parameter $\lambda(t)$ is bounded away from zero. Thus, Proposition \ref{P:new_modulation_alpha} gives
$$
|\alpha(t)-\alpha(t_n)|\lesssim \int_{t_n}^{t}|\dot{\alpha}(s)|\,ds\lesssim_M \int_{t_n}^{t}\delta(u(s)) \, ds
$$
so that
\begin{equation}\label{deltaalpha}
	\delta(u(t)) \sim |\alpha(t)| \lesssim_M |\alpha(t_n)| + e^{-ct_n}\lesssim_M \delta(u(t_n)) + e^{-ct_n}.
\end{equation}

This suffices to conclude that $\lim_{t\to\infty}\delta(u(t))=0$, so that that $[t_n,\infty)\subset I_0$ for all sufficiently large $n$. Consequently, the solution eventually never leaves the modulation region, and in fact (recalling \eqref{mod-bds-pt1}) the solution is exponentially close to a suitably modulated soliton. Moreover, the relation \eqref{Gron} combined with Propositions \ref{P:new_modulation} and \ref{P:new_modulation_alpha} guarantees the following:
\begin{equation}
|\alpha(t)| \to 0, \quad \lambda(t) \to \lambda_0 \in (0,\infty) \qtq{and} \theta(t) \to \theta_{0}\in \R / (2\pi \mathbb Z)\,\,\,\mbox{as}\,\,\,t \to \infty,
\end{equation}
with exponential convergence of all of these parameters.



Note that the solution $u$ under consideration must necessarily scatter backward in time. Indeed, if  
$\|u\|_{L^{10}_{t,x}((T^{-},0]\times \mathbb{R}^{3})} = +\infty$,  
then similar arguments imply that $\lim_{t\to -\infty} \delta(u(t)) = 0$. Therefore, by the modulated virial estimate in Proposition~\ref{ModVir} (which in this case also holds for negative times), we obtain
\[
\int_{-\infty}^{+\infty} \delta(u(t))\, dt = 0,
\]
which contradicts the assumption $\|u_{0}\|_{\dot H^{1}} < \|W\|_{\dot H^{1}}$.

To summarize the previous discussion, we have proved the following:

\begin{prop}[Convergence of nonscattering, constrained solutions]\label{prop_subcrit}
Let $u$ be a solution to \eqref{NLS} with maximal interval of existence $(T^-,T^+)$, such that $E(u_0) = E(W)$, $\|u_0\|_{\dot H^1} < \|W\|_{\dot H^1}$ and $\|u\|_{L^{10}_{t,x}([0,T^+)\times \R^3)} = \infty$. Then $T^+=\infty$ and there exist $\lambda_0>0$, $\theta_0 \in [0,2\pi)$ and $c>0$ such that
\begin{equation}\label{exp_conv_sub}
    \|u(t)-e^{-i\theta_0}\lambda_0^{-1/2}W(\lambda_0^{-1} \cdot )\|_{\dot{H}^1_x} \lesssim  e^{-ct}\,\,\,\mbox{for all}\,\,\,t\geq 0.
\end{equation}
Moreover, $\|u\|_{L^{10}_{t,x}((T^-,0]\times \R^3)} < +\infty$ and hence $u$ scatters backward in time.
\end{prop}

As described at the beginning of this section, with this result in hand, we can follow along standard arguments (as in \cite{DM_Thre, Liu2024arxiv})  to complete the proof of the main result, Theorem~\ref{thm:threshold}.


\vspace{0.5cm}
\noindent 
\textbf{Acknowledgments.} L.C. was partially supported by Conselho Nacional de Desenvolvimento Cient\'ifico e Tecnol\'ogico - CNPq grants 307733/2023-8 and 404800/2024-6 and Funda\c{c}\~ao de Amparo a Pesquisa do Estado de Minas Gerais - FAPEMIG grants APQ-03186-24 and APQ-03752-25. J.M. was partially supported by NSF grant DMS-2350225, Simons Foundation grant MPS-TSM-00006622 and CAPES grant 88887.937783/2024-00. L.G.F. was partially supported by Conselho Nacional de Desenvolvimento Cient\'ifico e Tecnol\'ogico - CNPq grant 307323/2023-4. L.C. and L.G.F. were partially supported by Coordena\c{c}\~ao de Aperfei\c{c}oamento de Pessoal de N\'ivel Superior - CAPES grant 88881.974077/2024-01 and CAPES/COFECUB grant 88887.879175/2023-00. This work was completed while L.C. was a visiting scholar at University of Oregon in 2025-26 under the support of Conselho Nacional de Desenvolvimento Cient\'ifico e Tecnol\'ogico - CNPq, for which the author is very grateful as it boosted the energy into the research project.

%
%
%
%
%
%
%

\addtocontents{toc}{\protect\vspace*{\baselineskip}}

\bibliographystyle{natbib} 



\begin{bibdiv}
\begin{biblist}

\bib{CamposCardoso2022}{article}{
      author={Campos, Luccas},
      author={Cardoso, Mykael},
       title={A virial-{M}orawetz approach to scattering for the non-radial
  inhomogeneous {NLS}},
        date={2022},
        ISSN={0002-9939,1088-6826},
     journal={Proc. Amer. Math. Soc.},
      volume={150},
      number={5},
       pages={2007\ndash 2021},
         url={https://doi.org/10.1090/proc/15680},
      review={\MR{4392336}},
}

\bib{CFMarxiv25}{article}{
      author={Campos, Luccas},
      author={Farah, Luiz~Gustavo},
      author={Murphy, Jason},
       title={Threshold solutions for the $3d$ cubic {INLS}: the
  energy-subcritical case},
     journal={In preparation},
}

\bib{CFR22}{article}{
      author={Campos, Luccas},
      author={Farah, Luiz~Gustavo},
      author={Roudenko, Svetlana},
       title={Threshold solutions for the nonlinear {S}chr\"{o}dinger
  equation},
        date={2022},
        ISSN={0213-2230},
     journal={Rev. Mat. Iberoam.},
      volume={38},
      number={5},
       pages={1637\ndash 1708},
         url={https://doi.org/10.4171/rmi/1337},
      review={\MR{4502077}},
}

\bib{CM_Threshold_INLS_2023}{article}{
      author={Campos, Luccas},
      author={Murphy, Jason},
       title={Threshold solutions for the intercritical inhomogeneous {NLS}},
        date={2023},
        ISSN={0036-1410},
     journal={SIAM J. Math. Anal.},
      volume={55},
      number={4},
       pages={3807\ndash 3843},
         url={https://doi.org/10.1137/22M1497663},
      review={\MR{4631010}},
}

\bib{CHL20}{article}{
      author={Cho, Yonggeun},
      author={Hong, Seokchang},
      author={Lee, Kiyeon},
       title={On the global well-posedness of focusing energy-critical
  inhomogeneous {NLS}},
        date={2020},
        ISSN={1424-3199},
     journal={J. Evol. Equ.},
      volume={20},
      number={4},
       pages={1349\ndash 1380},
         url={https://doi.org/10.1007/s00028-020-00558-1},
      review={\MR{4181951}},
}

\bib{CL21}{article}{
      author={Cho, Yonggeun},
      author={Lee, Kiyeon},
       title={On the focusing energy-critical inhomogeneous {NLS}: weighted
  space approach},
        date={2021},
        ISSN={0362-546X},
     journal={Nonlinear Anal.},
      volume={205},
       pages={Paper No. 112261, 21},
         url={https://doi.org/10.1016/j.na.2021.112261},
      review={\MR{4212088}},
}

\bib{DM08}{article}{
      author={Duyckaerts, Thomas},
      author={Merle, Frank},
       title={Dynamics of threshold solutions for energy-critical wave
  equation},
        date={2008},
        ISSN={1687-3017},
     journal={Int. Math. Res. Pap. IMRP},
       pages={Art ID rpn002, 67},
         url={https://doi.org/10.1093/imrp/rpn002},
      review={\MR{2470571}},
}

\bib{DM_Thre}{article}{
      author={Duyckaerts, Thomas},
      author={Merle, Frank},
       title={Dynamic of threshold solutions for energy-critical {NLS}},
        date={2009},
        ISSN={1016-443X},
     journal={Geom. Funct. Anal.},
      volume={18},
      number={6},
       pages={1787\ndash 1840},
         url={https://doi.org/10.1007/s00039-009-0707-x},
      review={\MR{2491692}},
}

\bib{DR10}{article}{
      author={Duyckaerts, Thomas},
      author={Roudenko, Svetlana},
       title={Threshold solutions for the focusing 3{D} cubic {S}chr\"{o}dinger
  equation},
        date={2010},
        ISSN={0213-2230},
     journal={Rev. Mat. Iberoam.},
      volume={26},
      number={1},
       pages={1\ndash 56},
         url={https://doi.org/10.4171/RMI/592},
      review={\MR{2662148}},
}

\bib{GM21}{article}{
      author={Guzm\'{a}n, Carlos~M.},
      author={Murphy, Jason},
       title={Scattering for the non-radial energy-critical inhomogeneous
  {NLS}},
        date={2021},
        ISSN={0022-0396},
     journal={J. Differential Equations},
      volume={295},
       pages={187\ndash 210},
         url={https://doi.org/10.1016/j.jde.2021.05.055},
      review={\MR{4268726}},
}

\bib{GuzmanXu2025}{article}{
      author={Guzm\'{a}n, Carlos~M.},
      author={Xu, Chengbin},
       title={Dynamics of the non-radial energy-critical inhomogeneous {NLS}},
        date={2025},
        ISSN={0926-2601,1572-929X},
     journal={Potential Anal.},
      volume={63},
      number={2},
       pages={601\ndash 630},
         url={https://doi.org/10.1007/s11118-024-10183-z},
      review={\MR{4953686}},
}

\bib{KV10}{article}{
      author={Killip, Rowan},
      author={Visan, Monica},
       title={The focusing energy-critical nonlinear {S}chr\"{o}dinger equation
  in dimensions five and higher},
        date={2010},
        ISSN={0002-9327},
     journal={Amer. J. Math.},
      volume={132},
      number={2},
       pages={361\ndash 424},
         url={https://doi.org/10.1353/ajm.0.0107},
      review={\MR{2654778}},
}

\bib{KV13}{incollection}{
      author={Killip, Rowan},
      author={Visan, Monica},
       title={Nonlinear {S}chr\"{o}dinger equations at critical regularity},
        date={2013},
   booktitle={Evolution equations},
      series={Clay Math. Proc.},
      volume={17},
   publisher={Amer. Math. Soc., Providence, RI},
       pages={325\ndash 437},
      review={\MR{3098643}},
}

\bib{LZ09}{article}{
      author={Li, Dong},
      author={Zhang, Xiaoyi},
       title={Dynamics for the energy critical nonlinear {S}chr\"{o}dinger
  equation in high dimensions},
        date={2009},
        ISSN={0022-1236},
     journal={J. Funct. Anal.},
      volume={256},
      number={6},
       pages={1928\ndash 1961},
         url={https://doi.org/10.1016/j.jfa.2008.12.007},
      review={\MR{2498565}},
}

\bib{LinZeng}{article}{
      author={Lin, Zhiwu},
      author={Zeng, Chongchun},
       title={Instability, index theorem, and exponential trichotomy for linear
  {H}amiltonian {PDE}s},
        date={2022},
        ISSN={0065-9266},
     journal={Mem. Amer. Math. Soc.},
      volume={275},
      number={1347},
       pages={v+136},
         url={https://doi.org/10.1090/memo/1347},
      review={\MR{4352468}},
}

\bib{LiuMiaoZheng2025}{article}{
      author={Liu, Xuan},
      author={Miao, Changxing},
      author={Zheng, Jiqiang},
       title={Global well-posedness and scattering for mass-critical
  inhomogeneous {NLS} when {$d\ge 3$}},
        date={2025},
        ISSN={0025-5874,1432-1823},
     journal={Math. Z.},
      volume={311},
      number={3},
       pages={Paper No. 54, 38},
         url={https://doi.org/10.1007/s00209-025-03855-y},
      review={\MR{4958518}},
}

\bib{Liu2024arxiv}{article}{
      author={Liu, Xuan},
      author={Yang, Kai},
      author={Zhang, Ting},
       title={Dynamics of threshold solutions for the energy-critical
  inhomogeneous {NLS}},
        date={2024},
     journal={arXiv preprint arXiv:2409.00073},
}

\bib{MiaoMurphyZheng2021}{article}{
      author={Miao, Changxing},
      author={Murphy, Jason},
      author={Zheng, Jiqiang},
       title={Scattering for the non-radial inhomogeneous {NLS}},
        date={2021},
        ISSN={1073-2780,1945-001X},
     journal={Math. Res. Lett.},
      volume={28},
      number={5},
       pages={1481\ndash 1504},
         url={https://doi.org/10.4310/mrl.2021.v28.n5.a9},
      review={\MR{4471717}},
}

\bib{Murphy2022}{article}{
      author={Murphy, Jason},
       title={A simple proof of scattering for the intercritical inhomogeneous
  {NLS}},
        date={2022},
        ISSN={0002-9939,1088-6826},
     journal={Proc. Amer. Math. Soc.},
      volume={150},
      number={3},
       pages={1177\ndash 1186},
         url={https://doi.org/10.1090/proc/15717},
      review={\MR{4375712}},
}

\bib{YZZ22}{article}{
      author={Yang, Kai},
      author={Zeng, Chongchun},
      author={Zhang, Xiaoyi},
       title={Dynamics of threshold solutions for energy critical {NLS} with
  inverse square potential},
        date={2022},
        ISSN={0036-1410},
     journal={SIAM J. Math. Anal.},
      volume={54},
      number={1},
       pages={173\ndash 219},
         url={https://doi.org/10.1137/21M1406003},
      review={\MR{4358028}},
}

\end{biblist}
\end{bibdiv}

\newcommand{\Addresses}{{
  \bigskip
  \footnotesize

  L. Campos, \textsc{Department of Mathematics, Universidade Federal de Minas Gerais, Brazil}\par\nopagebreak
  \textit{E-mail address:} \texttt{luccas@ufmg.br}\\
  
L. G. Farah, \textsc{Department of Mathematics, Universidade Federal de Minas Gerais, Brazil}\par\nopagebreak
  \textit{E-mail address:} \texttt{farah@mat.ufmg.br}\\
  
J. Murphy, \textsc{Department of Mathematics, University of Oregon, Eugene, OR, USA}\par\nopagebreak
  \textit{E-mail address:} \texttt{jamu@uoregon.edu}  
}}
\setlength{\parskip}{0pt}
\Addresses
\batchmode

\end{document}